\documentclass[12pt]{amsart}
\usepackage{amssymb,amsmath}
\usepackage{geometry}
\usepackage{mathrsfs}

\usepackage{graphicx}

\newtheorem{theorem}{Theorem}[section]
\newtheorem{lemma}[theorem]{Lemma}

\newtheorem{corollary}[theorem]{Corollary}
\theoremstyle{definition}

\newtheorem{example}[theorem]{Example}

\newcommand{\vz}{\underline{z}}
\newcommand{\va}{\underline{a}}

\newcommand{\NN}{\mathbb{N}}

\begin{document}

\title[An improved bound in Wirsing's problem]{An improved bound in Wirsing's problem}

\author{Dmitry Badziahin, Johannes Schleischitz}


\thanks{Middle East Technical University, Northern Cyprus Campus, Kalkanli, G\"uzelyurt \\
    johannes.schleischitz@univie.ac.at}


\begin{abstract}
    We improve the lower bound for the classical exponent of approximation
    $w_{n}^{\ast}(\xi)$ connected to Wirsing's famous problem of
    approximation to real numbers by algebraic numbers of degree at most
    $n$. Our bound exceeds $n/\sqrt{3}\approx 0.5773n$ and thus provides a
    reasonable qualitative improvement to previous bounds of order
    $n/2+O(1)$. We further establish new relations between several
    classical exponents of approximation.
\end{abstract}

\maketitle

{\footnotesize{

Math Subject Classification 2010: 11J13, 11J82, 11J83}}

\vspace{1mm}

\section{Wirsing's problem: Introduction and main results}  \label{int01}
In this paper we are concerned with approximation to a
transcendental real number $\xi$ by algebraic real numbers $\alpha$
of degree at most $n$. A classical setup is to relate the quality of
approximation $|\xi-\alpha|$ with the naive height $H(\alpha)$ of
the minimal polynomial of $\alpha$ over $\mathbb{Z}$ with coprime
coefficients, that is the maximum modulus of its coefficients. In
1961 Wirsing~\cite{wirsing} defined the quantity $w_{n}^{\ast}(\xi)$
as the supremum of $w^{\ast}$ for which the estimate
\[
| \xi-\alpha| < H(\alpha)^{-w^{\ast}-1}
\]
has infinitely many solutions in algebraic real numbers $\alpha$ of
degree at most $n$.

A longstanding open problem posed by Wirsing in~\cite{wirsing} is to decide
whether the quantity $w_{n}^{\ast}(\xi)$ is always bounded from below by $n$.
For $n=1$ this is true by Dirichlet's Theorem. In fact, by the theory of
continued fractions, the estimate $|\alpha-\xi|< cH(\alpha)^{-2}$ has
infinitely many solutions in rational numbers $\alpha=p/q$ (s.t.
$H(\alpha)=\max\{|p|,|q|\}$) for any $c>\max\{1,|\xi|\}/\sqrt{5}$,
see~\cite[Theorem~2F in Chapter I]{schmidt}. It was further verified for
$n=2$ in a paper of Davenport and Schmidt~\cite{davsh67} from 1967, who
similarly established an estimate of the form $|\alpha-\xi|< cH(\alpha)^{-3}$
with some explicit $c=c(\xi)$ for infinitely many rational or quadratic
irrational numbers. In fact, the numbers $\alpha$ can be chosen
quadratic irrationalities~\cite{moscj}. Furthermore, a combination of
Sprind\v{z}uk's famous result~\cite{sprindzuk} with ~\cite[(7)]{wirsing} implies
that almost all $\xi$ with respect to Lebesgue measure satisfy the identity
$w_{n}^{\ast}(\xi)=n$ for any $n\geq 1$. The identity also holds for any
algebraic number $\xi$ of degree larger than $n$ by an application of Schmidt
Subspace Theorem~\cite[Theorem~2.9]{bugbuch}. Apart from that, for $n\geq 3$
and general $\xi$, Wirsing's problem remains open.

It should be mentioned that a similar problem with respect to
approximation by algebraic integers was answered negatively. For the
case of approximation by cubic algebraic integers counterexamples
were found by Roy~\cite{roy}.

A bound of the form $w_{n}^{\ast}(\xi)\geq n/2+1-o(1)$ as
$n\to\infty$ was established by Wirsing himself in the same
paper~\cite{wirsing}. This had so far only been mildly improved by
some additive constant. Bernik and Tsishchanka~\cite{ber} were first
to improve the bound to an expression of order $n/2+2-o(1)$, and
this was refined in follow up papers by Tsishchanka, the
latest~\cite{Tsi07} contains the best currently known bound of order
$n/2+3-o(1)$ (as $n$ tends to infinity). In this paper we finally go
beyond the bound of order $n/2+O(1)$ by establishing the estimate
$w_{n}^{\ast}(\xi)/n>1/\sqrt{3}>0.57$. To state our main results in
a compact form let us define
\[
\overline{w}^{\ast}(\xi)= \limsup_{n\to\infty} \frac{w_{n}^{\ast}(\xi)}{n}, \qquad \underline{w}^{\ast}(\xi)= \liminf_{n\to\infty} \frac{w_{n}^{\ast}(\xi)}{n}.
\]
Then we show

\begin{theorem}  \label{1B}
    Let $n\geq 4$ be an integer. Let $\xi$ be any transcendental real number. Then
    we have
    \[
    w^{\ast}_{n}(\xi) > \frac{1}{\sqrt{3}}\cdot n= 0.5773\ldots n.
    \]
    In particular $\underline{w}^{\ast}(\xi) \geq 1/\sqrt{3}$.
    Moreover,
    \begin{equation} \label{eq:obenba}
    \overline{w}^{\ast}(\xi) \geq \delta,
    \end{equation}
    where $\delta=0.6408\ldots$ is given as $G(\gamma_{0})$ where
    \begin{equation} \label{eq:G}
    G(t)= \frac{4(t-t^{2})}{2t^{2}+2t-1+\sqrt{4t^{4}+24t^{3}-32t^{2}+12t+1}}
    \end{equation}
    and $\gamma_{0}$ is the root of $Q(t)=4t^4 - 12t^3 + 10t^2 - 6t + 1$
    in $t\in(0,1/2)$.
\end{theorem}

Clearly if Wirsing's problem has a positive answer then
$\underline{w}^{\ast}(\xi)\geq 1$ for any transcendental real $\xi$.
However, it seems that there is no easy argument available to deduce
any lower bound better than $1/2$ even for the larger quantity
$\overline{w}^{\ast}(\xi)$.
Theorem~\ref{1B} follows from optimization of $m$ in the following
result.

\begin{theorem} \label{1A}
    Let $n\geq 4$ be an integer
    and $\xi$ be a transcendental real number. Then for any $1\leq m<
    (n-1)/2$ one has
    \[
    w_{n}^{\ast}(\xi) \geq \frac{4mn+6n-4m^{2}-8m}{2m+2-n+\sqrt{n^{2}+12mn+20n-12m^{2}-24m+4}}.
    \]
\end{theorem}

%

A slight improvement of the bound can be derived by our method.
However, the resulting bound is a root of a complicated
cubic polynomial and
the refinement is too insignificant to improve on the factors
$1/\sqrt{3}$ and $\delta$ of Theorem~\ref{1B}.
See the comments below the proof for details. It is also worth
noting that for $n\leq 24$, the bound by Tsishchanka~\cite{Tsi07}
for $w_{n}^{\ast}(\xi)$ is better. The table below compares the
bound of~\cite{Tsi07} with those from Theorem~\ref{1A} with suitable $m$
for some
particular values of $n$.

\begin{center}
    \begin{tabular}{ |c|c|c| }
        \hline
        n & Tsi & BS \\ \hline
        3 & 2.73 & - \\
        4 & 3.45 & 2.64 \\
        5 & 4.14 & 3.34 \\
        10 & 7.06 & 6.42 \\
        20 & 12.39 & 12.16   \\
        24 & 14.46 & 14.46   \\
                25 & 14.98 & 15.04   \\
        30 & 17.55 & 17.92   \\
        50 & 27.70 & 29.46 \\
        100 & 52.84 & 58.32 \\
        1000 & 502.98 & 577.92 \\
        \hline
    \end{tabular}
\end{center}

While other approaches to Wirsing's problem rely on counting
algebraic numbers in small intervals, see for instance the recent
preprint by Bernik, Goetze and Kalosha~\cite{bgk}, our result relies
solely on relations between different exponents of Diophantine
approximation defined in Section~\ref{e} below. Thereby we build up on ideas
of Wirsing~\cite{wirsing}, Davenport and Schmidt~\cite{davsh} and
Laurent~\cite{laurent}.

For variants of Wirsing's problem that have been studied,
including prescribing
the degree of $\alpha$ as equal to $n$ (see~\cite{buteu}) or considering
algebraic integers $\alpha$ of degree $n+1$ as in~\cite{roy}, our method
does not apply. The concrete obstruction is identified in Section~\ref{twor}.
Nevertheless we conjecture that the claims remain true.

\section{Other classical exponents of approximation}

\subsection{Exponents of Diophantine approximation} \label{e}

 Apart from $w_{n}^{\ast}(\xi)$ itself, the most important exponents in this paper are $\widehat{\lambda}_{n}(\xi)$,
 defined as the supremum of $\lambda$ such that the inequalities
 \begin{equation} \label{eq:tierekt}
 1\leq x\leq X, \qquad
 L(\underline{x}):=\max_{1\leq j\leq n} \vert \zeta^{j}x-y_{j}\vert\leq X^{-\lambda},
 \end{equation}
 have an integer vector solution $\underline{x}=(x,y_{1},\ldots,y_{n})$ for all large
 $X$. An easy application of the Dirichlet's theorem implies that $\widehat{\lambda}_{n}(\xi)$ is bounded below by
 $1/n$. On the other hand, Davenport and Schmidt~\cite{davsh} verified that
 $\widehat{\lambda}_{n}(\xi)$ does not exceed $2/n$. Thus it may vary only up to a factor $2$.
 Slight improvements of the upper bound for odd $n$
 by Laurent~\cite{laurent}
 and for even $n$ by Schleischitz~\cite{equprin, indag, js} were obtained later.
 See also Roy~\cite{roy3} for $n=3$. Note that
 it follows from Davenport and Schmidt~\cite[Lemma~1]{davsh}
  that any improvement of the factor $2$ separating the upper bound
  from the trivial lower bound $1/n$
  would directly
 lead to an improvement of the factor $1/2$ in the Wirsing's problem (as we
 establish in this paper), see Section~\ref{twor}.
 While we are unable to provide such improvements for $\widehat{\lambda}_{n}(\xi)$,
 the underlying estimate of Davenport and Schmidt is a crucial ingredient in our argument.

 We will sporadically make reference to the ordinary exponents
 $\lambda_{n}(\xi)$ defined similarly, but where we
 impose that \eqref{eq:tierekt} has a solution for some arbitrarily large
 values of $X$. This weaker condition is reflected in $\lambda_{n}(\xi)\geq \widehat{\lambda}_{n}(\xi)$.
 We will further employ the dual linear form exponents
 $w_{n}(\xi), \widehat{w}_{n}(\xi)$ defined as the supremum
 of $w$ so that the system
 \[
 1\leq \max_{1\leq j\leq n} |a_{j}|\leq X, \qquad |a_{0}+\xi a_{1}+\cdots+\xi^{n}a_{n}|\leq X^{-w}
 \]
 has a solution in integers $a_{0},\ldots,a_{n}$ for arbitrarily
 large $X$ and all sufficiently large $X$, respectively. These exponents also
 satisfy $w_{n}(\xi)\geq \widehat{w}_{n}(\xi)\geq n$ by the Dirichlet box principle,
 and again in~\cite{davsh} Davenport and Schmidt found the upper bound
 $2n-1$ for the uniform exponent $\widehat{w}_{n}(\xi)$, as well as an improved
 bound for $n=2$ which turned out to be sharp~\cite{royjl}. Again, as for
 $\widehat{\lambda}_{n}(\xi)$, the upper and lower bounds roughly differ by a
 factor of 2 which for large $n$ has not been improved so far. However,
 refinements in the constant term were made first
 by Bugeaud and Schleischitz~\cite{buschlei}. The proof strategy
 in~\cite{buschlei}, in the light of later
 findings~\cite{unif, mamo}, in turn yields slightly stronger bounds,
 in particular $\widehat{w}_{n}(\xi)\le 2n-2$
 for $n\ge 10$.
 See also~\cite{acta2018}, where a conjectural bound of
 order $(1+1/\sqrt{2})n-o(1)<1.71n$ was motivated as well.
 Again, while we do not improve the bounds for the exponent $\widehat{w}_{n}(\xi)$, another estimate from~\cite{buschlei} linking
 it with $w_{n}^{\ast}(\xi)$
 is essential for this paper.

\subsection{New relations between classical exponents}

On the way to the main results we establish the following connections between
various exponents of approximation which are of some
independent interest.

\begin{theorem}  \label{openup2}
    Let $m,n,\xi$ be as in Theorem~\ref{1A} and assume
    \begin{equation}  \label{eq:only2}
    \widehat{\lambda}_{n}(\xi) > \frac{1}{n-m}.
    \end{equation}
    Then we have
\begin{equation} \label{eq:h11}
\widehat{w}_{n-m}(\xi)\geq
\frac{(n-m)\widehat{\lambda}_{n}(\xi)+n-2m-1}{1-m\widehat{\lambda}_{n}(\xi)}.
\end{equation}
Moreover
	\begin{equation} \label{eq:h21}
	w_{n-m}(\xi)\geq \max\left\{\frac{(n-m)\widehat{\lambda}_{n}(\xi)+n-2m-2}{1-(m+1)\widehat{\lambda}_{n}(\xi)},\; \frac{(n-m)\lambda_{n}(\xi)+n-2m-1}{1-m\lambda_{n}(\xi)}\right\},
	\end{equation}
	and conversely
	\begin{equation} \label{eq:verynew}
	w_{n-m}(\xi) \leq \frac{n-m-1}{m+1}\cdot \frac{(n-m)\widehat{\lambda}_{n}(\xi)+n-2m-1}{(n-m)\widehat{\lambda}_{n}(\xi)-1}.
	\end{equation}
    Finally,
    \begin{equation} \label{eq:windag}
    w_{m+1}(\xi) \leq \frac{1}{\widehat{\lambda}_{n}(\xi)}<n-m.
    \end{equation}
\end{theorem}

In fact, we only require \eqref{eq:h11} for the proof of Theorems~\ref{1B}
and~\ref{1A}. The bounds \eqref{eq:h11}, \eqref{eq:h21} are increasing in
$\widehat{\lambda}_{n}(\xi)$ and non-trivial (i.e. exceed $n-m$). We remark
that a very similar argument would lead to the estimate $w_{\lfloor
n/2\rfloor}(\xi)\leq 1/\widehat{\lambda}_{n}(\xi)$ if
$\widehat{\lambda}_{n}(\xi)>\lceil n/2\rceil^{-1}$ (note that it is an upper
bound here), which leads to a contradiction, in view of the reverse estimate
$w_{\lfloor n/2\rfloor}(\xi)\geq \lfloor n/2\rfloor$. This would yield an
alternative proof of the bounds for $\widehat{\lambda}_{n}(\xi)$
in~\cite{laurent}. We want to state the special case $n=4, m=1$ where much
cancellation occurs as a corollary.

\begin{corollary} \label{4fall}
    Let $\xi$ be real transcendental with $\widehat{\lambda}_{4}(\xi)>1/3$.
    We have
	\[
	\frac{3\widehat{\lambda}_{4}(\xi)+1}{3\widehat{\lambda}_{4}(\xi)-1}\geq w_{3}(\xi) \geq \max\left\{ \frac{3\widehat{\lambda}_{4}(\xi)}{1-2\widehat{\lambda}_{4}(\xi)}, \frac{3\lambda_{4}(\xi)+1}{1-\lambda_{4}(\xi)}\right\},
	\qquad \widehat{w}_{3}(\xi)\geq
	\frac{3\widehat{\lambda}_{4}(\xi)+1}{1-\widehat{\lambda}_{4}(\xi)},
	\]    %
    and
    \[
    w_{2}(\xi)\leq \frac{1}{\widehat{\lambda}_{4}(\xi)}<3.
    \]
\end{corollary}

Comparing the left lower and the upper bound for $w_{3}(\xi)$ gives
$\widehat{\lambda}_{4}(\xi)\leq (\sqrt{19}+2)/15=0.4239\ldots$, which is
however weaker than the best known bound $0.3706\ldots$ from~\cite{js}
(a weaker bound
in~\cite{equprin} differs only in the fifth decimal digit).
The same method can be applied to any even $n$ and $m=n/2-1$. Then
Theorem~\ref{openup2} yields that in the case $\widehat{\lambda}_{n}(\xi) >
\frac{2}{n+2}$ one has
\[
\max\left\{\frac{(n+2)\widehat{\lambda}_{n}(\xi)}{2-n\widehat{\lambda}_{n}(\xi)},\; \frac{(n+2)\lambda_{n}(\xi)+2}{2-(n-2)\lambda_{n}(\xi)}\right\}\leq w_{\frac{n}{2}+1}(\xi) \leq \frac{(n+2)\widehat{\lambda}_{n}(\xi)+2}{(n+2)\widehat{\lambda}_{n}(\xi)-2}.
\]
This further implies $\widehat{\lambda}_{n}(\xi)\leq 2/n-(4/3+o(1))n^{-2}$ as
$n\to\infty$, however again larger than the bound
in~\cite[Theorem~4.1]{equprin} of order $2/n-(3.18\ldots+o(1))n^{-2}$.

As for the exponent $w_{n}^{\ast}$, we define the upper limits
\[
\overline{\widehat{w}}(\xi) = \limsup_{n\to\infty} \frac{\widehat{w}_{n}(\xi)}{n}, \qquad \overline{\widehat{\lambda}}(\xi) = \limsup_{n\to\infty} n\widehat{\lambda}_{n}(\xi),
\]
and accordingly, the lower limits $\underline{\widehat{w}}(\xi)$
and $\underline{\widehat{\lambda}}(\xi)$. These
quantities all lie in the interval $[1,2]$, see Section~\ref{e}. Another
consequence of Theorem~\ref{openup2} reads
\begin{corollary} \label{comp}
    For any transcendental real number $\xi$ we have
    \[
    \overline{\widehat{w}}(\xi) \geq \frac{1-2\cdot R\left(\overline{\widehat{\lambda}}(\xi) \right)}{1-\left(\overline{\widehat{\lambda}}(\xi)+1\right)\cdot R\left(\overline{\widehat{\lambda}}(\xi) \right) +\overline{\widehat{\lambda}}(\xi)\cdot R\left(\overline{\widehat{\lambda}}(\xi) \right)^{2}} =: S\left(\overline{\widehat{\lambda}}(\xi) \right), \]
    and similarly
    \[
    \underline{\widehat{w}}(\xi) \geq \frac{1-2\cdot R\left(\underline{\widehat{\lambda}}(\xi) \right)}{1-\left(\underline{\widehat{\lambda}}(\xi)+1\right)\cdot R\left(\underline{\widehat{\lambda}}(\xi) \right) +\underline{\widehat{\lambda}}(\xi)\cdot R\left(\underline{\widehat{\lambda}}(\xi) \right) ^{2}} =: S\left(\underline{\widehat{\lambda}}(\xi) \right)
    \]
    where the function $R(t)$ is given as
    \[
    R(t)= \frac{t-\sqrt{2t-t^{2}}}{2t}.
    \]
\end{corollary}

One can verify that the function $S$ induces an increasing bijection from the
interval $[1,2]$ to itself. We compute $S(1.5)=1.0718\ldots$,
$S(1.75)=1.2038\ldots$, $S(1.99)=1.7527\ldots$, $S(1.9999)=1.9721\ldots$.
Corollary~\ref{comp} complements~\cite[Theorem~3.4]{equprin} where reverse
estimates in form of lower bounds for $\underline{\widehat{\lambda}}(\xi),
\overline{\widehat{\lambda}}(\xi)$ in terms of $\underline{\widehat{w}}(\xi),
\overline{\widehat{w}}(\xi)$ respectively were established (formulated there
for ordinary exponents but as stated below the theorem it is true for uniform
exponents as well).

\begin{proof}
    Let $\epsilon>0$ be arbitrary. Then for $n$ large enough we get
    $\lambda:= n\widehat{\lambda}_{n}(\xi) >
    \overline{\widehat{\lambda}}(\xi) - \epsilon$. Fix some $\alpha\in [0,1/2)$ and select $m:= \lfloor n\alpha\rfloor$.
    Define $c$ from the
    equation $\widehat{\lambda}_{n}(\xi)=c/(n-m)$. Then $c =
    (1-\alpha)\lambda + o(1)$ as $n\to\infty$. If $c$ exceeds one, we may apply Theorem~\ref{openup2} to get
    \[
    \frac{\widehat{w}_{n-m}(\xi)}{n-m} \geq \frac{\widehat{\lambda}(\xi) + \frac{n-2m-1}{n-m}}{1-m\widehat{\lambda}(\xi)}=
     \frac{1-2\alpha}{1-(\lambda+1)\alpha+\lambda\alpha^{2}}+o(1), \qquad n\to\infty.
    \]
    In the given range of $\alpha$ the expression is maximized for
    \[
    \alpha= \frac{1}{2}-\frac{\sqrt{2\lambda-\lambda^{2}}}{2\lambda}+o(1),
    \qquad n\to\infty,
    \]
    and inserting gives the first lower bound of the corollary as we may choose $\lambda$ arbitrarily close to $\overline{\widehat{\lambda}}(\xi)$. Finally
    we check that the prerequisite $c>1$ is equivalent to $\lambda+\sqrt{2\lambda-\lambda^{2}}>2$ for small enough $\epsilon$ and large enough $n$. This inequality is verified for $\lambda\in(1,2)$,
    and for $\lambda=1$ our claim holds for trivial reasons. The
    second lower bound follows analogously.
    \end{proof}

\section{Preparatory concepts and the crucial lemma} \label{prepara}

In this section we prepare the proof of Theorem~\ref{openup2}.

\subsection{Minimal points and the key lemma}

We will use the concept of minimal points as for instance used
in~\cite{davsh, laurent}. Let $n\in\mathbb{N}$ and transcendental
real $\xi$ be
given. Consider the simultaneous approximation problem \eqref{eq:tierekt}.
Then $n,\xi$ give rise to a unique (up to sign) sequence of best
approximations
\[
\underline{x}_{i}= \underline{x}_{i}(n,\xi)= (x_{i,0},x_{i,1},\ldots,x_{i,n}), \qquad i\geq 1,
\]
with the property that $L(\underline{x}_{i})$
minimizes $L(\underline{x})$ upon all integer vectors
$\underline{x}=(x,y_{1},\ldots,y_{n})$ with $1\leq x\leq x_{i,0}$.
They have the properties
\[
x_{1,0}< x_{2,0}< \ldots, \qquad L(\underline{x}_{1}) > L(\underline{x}_{2}) > \cdots.
\]
The study of the sequence of minimal points is the basis of many results
regarding exponents of approximation, and we will make use of this concept in
the following key lemma whose proof is an adaption of the method by
Laurent~\cite{laurent}.

\begin{lemma} \label{le1A}
    Let $n\geq 2$ be an integer and $1\leq m\leq \lceil n/2\rceil$ another integer. Let $\xi$ be a transcendental real number that satisfies
    \begin{equation}  \label{eq:only}
    \widehat{\lambda}_{n}(\xi) > \frac{1}{n-m+1}.
    \end{equation}
    Then for any large $i$ the vectors
    \[
    (x_{i,0}, \ldots, x_{i,n-m}),\; (x_{i,1},\; \ldots,\; x_{i,n-m+1}),\; \ldots,\;
    (x_{i,m},\; \ldots,\; x_{i,n})
    \]
    formed from the $i$-th best approximation are linearly independent.
\end{lemma}

It is worth pointing out that the lemma uses a slightly relaxed
restriction on $m$ compared to Theorems~\ref{1A}, \ref{openup2}. The proof of
Theorem~\ref{openup2} will demonstrate that in contrast to the lemma,
indeed some of its claims cannot be extended to the cases $m=\lfloor n/2\rfloor$ or $m=\lceil n/2\rceil$.

\begin{proof}
    Fix the values $n,m$ and $\xi$ from the lemma.
    We follow the proof of Laurent~\cite{laurent}.
    Let $\lambda\in (1/(n-m+1), \widehat{\lambda}_{n}(\xi))$ be fixed for
    now to be specified later.
    For every $i\geq 1$,
    let $h=h_{i}$ be the smallest integer for
    which the $(h+1)\times (n-h+1)$ Hankel matrix
    \[
    V_{i}(h)=\begin{pmatrix}
    x_{i,0} & x_{i,1} & \cdots & x_{i,n-h}\\
    x_{i,1} & x_{i,2} & \cdots & x_{i,n-h+1}\\
    \ddots & \ddots & \ddots & \ddots \\
    x_{i,h} & x_{i,h+1} & \cdots & x_{i,n}
    \end{pmatrix}, \qquad\qquad i\geq 1,
    \]
    has rank at most $h$ (i.e. not full rank). Then the vectors
    $\vz_j:=(x_{i,j},x_{i,j+1}.\ldots, x_{i,j+h-1})$, $j\in\{0,\ldots, n-h+1\}$
    satisfy the recurrence relations
    $$
    a_0\vz_j + a_1\vz_{j+1} + \cdots + a_h\vz_{j+h} = \underline{0}.
    $$
    Now \cite[Lemma~1]{laurent} implies that one can choose integer
    coefficients $a_j$ such that $$\max\{|a_0|, \ldots, |a_h|\} \ll
    Z^{1/(n-2h+2)},$$ where $Z$ denotes the maximum of the absolute values of
    all the $h\times h$ determinants formed from any $h$ of the vectors
    $\vz_j$. On the other hand, by subtracting the first row of such matrix,
    multiplied by $\xi^j$, from the $j$'th row of this matrix, we can verify
    that for large $i$, $Z= o\left(x_{i,0}^{1-(h-1)\lambda}\right)$. Then it is easy to
    check that for
    \[
    \lambda > \frac{1}{n-h_{i}+1}
    \]
    one has $\max\{|a_0|, \ldots, |a_h|\}  = o(x_{i,0}^\lambda).$

    Consider the polynomial $P_i(z):= a_0+a_1z+\ldots + a_h z^h$. One
    notices that
    $$
    |x_{i,0} P_i(\xi)| = |a_1(x_{i,0}\xi - x_{i,1}) + a_2(x_{i,0}\xi^2 - x_{i,2}) + \ldots + a_h(x_{i,0}\xi^h-x_{i,h})| = o(1).
    $$
    Now consider $|a_0 x_{i-1,0} + a_1x_{i-1,1} + \cdots + a_hx_{i-1,h}|$. As
    before, it equals
    $$
    |x_{i-1,0}P_i(\xi) + a_1(x_{i-1,1} - x_{i-1,0}\xi) + \cdots + a_h(x_{i-1,h} - x_{i-1,0}\xi^h)| = o(1),
    $$
and because it must be an integer, we have that for $i\ge i_0$ large enough,
this expression equals zero.

    Now suppose there exists an infinite strictly increasing sequence
    $(i_k)_{k\in \NN}$ of indices such that $i_1\ge i_0$ and $h_{i_k}\le m$. We obtain that there is an integer vector $\underline{a}_{k}= (a_0, a_1, \ldots,
    a_h)$ which annihilates both matrices $V_{i_k}(h)$ and $V_{i_k-1}(h)$,
    i.e. $V_{i_k}(h)\cdot \underline{a}_{k}= V_{i_k-1}(h)\cdot
    \underline{a}_{k}=\underline{0}$. In particular, it implies that
    $h_{i_k-1}\le h_{i_k}$. By applying the same arguments iteratively to $V_{i_k-1}(h)$, $V_{i_k-2}(h)$ and so on, we get that $h_i\le m$ and $h_{i-1}\le h_i$ for all $i\ge i_0$. Since $h_i$ can not be arbitrarily large, the sequence $h_i$ is ultimately
    constant. In other words, for large $i$ we have $h_i = h\le m$. We further
    derive that for such large $i$ there is a vector $\va$ which does not depend on $i$ and
    annihilates all matrices $V_i(h)$. But that means there is a linear dependence between
    $1,\xi, \xi^2,\ldots, \xi^h$ which contradicts the assumption
    that $\xi$ is transcendental.
    \end{proof}

\begin{example}
    Let $n\geq 2$ and $m=1$. The lemma claims that as soon as $\widehat{\lambda}_{n}(\xi)>1/n$ the vectors $(x_{i,0}, x_{i,1}, \ldots, x_{i,n-1})$ and $(x_{i,1}, x_{i,2}, \ldots, x_{i,n})$ are linearly independent for large $i$. Notice that the condition is necessary. Indeed, any
    number $\xi$ with $\lambda_{1}(\xi)>2n-1$ (or equivalently $\lambda_{n}(\xi)>1$, see~\cite[Theorem~1.6]{ich})
    has infinitely many $\underline{x}_{i}$ with constant ratios $x_{i,j+1}/x_{i,j}, 0\leq j\leq n-1$, and thus the claim fails, see~\cite[Lemma~2.3]{ich} and
    also~\cite[Lemma~1]{bu}. Lemma~\ref{le1A} thus gives a new proof
    that such numbers satisfy $\widehat{\lambda}_{n}(\xi)=1/n$ (this statement is
    already contained in~\cite[Theorem~1.6]{ich}).
\end{example}

We state an easy consequence of the lemma.

\begin{corollary} \label{ccoo}
    Let $m,n,\xi$ be as in Lemma~\ref{le1A}, and assume \eqref{eq:only} holds.
    Then for any
    $\lambda< \widehat{\lambda}_{n}(\xi)$
    and any large $X$ the system
    \begin{equation} \label{eq:dassystem}
    1\leq |x|\leq X, \qquad \quad \max_{1\leq j\leq n-m} |x\xi^{j}-y_{j}| \leq X^{-\lambda}
    \end{equation}
    has $m+1$ linearly independent solutions in integer vectors
    $(x,y_{1},\ldots,y_{n-m})$. Similarly, for any $\lambda< \lambda_{n}(\xi)$ the system \eqref{eq:dassystem}
	has $m+1$ linearly independent solutions for some arbitrarily large $X$ .
\end{corollary}

\begin{proof}
        Without loss of generality assume $\xi>0$ to avoid writing absolute values. Let
        \[
        c=\frac{1}{2\max\{1,\xi^{n}\}(1+(\xi+\xi^{-1})^{n-1})}.
        \]
    Let $\lambda< \widehat{\lambda}_{n}(\xi)$ be fixed for the moment.
    Let $X>0$ be arbitrary large. For simplicity define the auxiliary parameter $Y=X/(2\max\{1,\xi^{n}\})$. Then the system
    \[
    1\leq x\leq Y, \qquad \max_{1\leq j\leq n} | x\xi^{j}-y_{j}| < cY^{-\lambda}
    \]
    has a solution in integers $(x,y_{1},\ldots,y_{n})\in \mathbb{Z}^{n+1}$ which can be chosen one of the best approximation vectors $\underline{x}_{i}(n,\xi)$.
    For $1\leq i\leq m$ and $1\leq j\leq n-m$ we have
    \begin{align*}
    | y_{i}\xi^{j}-y_{i+j}|&= | (y_{i}\xi^{j}- x\xi^{i+j}) + (x\xi^{i+j}- y_{i+j})|
    \leq cY^{-\lambda}(1+\xi^j)\\
    &\leq \frac{Y^{-\lambda}}{2\max\{1,\xi^{n}\}}\leq \frac{Y^{-\lambda}}{(2\max\{1,\xi^{n}\})^{\lambda}}=X^{-\lambda}.
    \end{align*}
    We conclude
    \[
    \max_{1\leq j\leq n-m} | x\xi^{j}-y_{j}|  < X^{-\lambda}, \qquad \max_{1\leq j\leq n-m} | y_{i}\xi^{j}-y_{i+j}|< X^{-\lambda}
    \]
    and
    \[
    \max\{|x|,|y_{1}|,\ldots,|y_{n-m}|\} < 2\max\{1,\xi^{n}\} Y=X.
    \]
    This shows that the vectors $(x,y_{1},\ldots,y_{n-m}),
    (y_{1},y_{2}\ldots,y_{n-m+1}),
    \ldots$, $(y_{m},y_{m+1},\ldots,y_{n})$ satisfy the
    estimates~\eqref{eq:dassystem}. Moreover
    they are linearly independent by Lemma~\ref{le1A}.
    The first claim follows. The second claim on the ordinary exponents
	$\lambda_{n}(\xi)$ is derived very similarly by
	considering minimal points $(x,y_{1},\ldots,y_{n})$ as in the
	definition of $\lambda_{n}(\xi)$ and putting $X=x$.
    \end{proof}

Let $N\geq 1$ be an integer. For any integer $l\in \{1, 2,\ldots, N+1\}$, define the
successive minima exponents $\lambda_{N,l}(\xi)$ as the supremum of $\lambda$
so that
\[
1\leq x\leq X, \qquad \max_{1\leq i\leq N}|x\xi^{i}-y_{i}|< X^{-\lambda}
\]
has $l$ linearly independent solution vectors $(x,y_{1},\ldots,y_{N})$ for
arbitrarily large $X$. Similarly define $\widehat{\lambda}_{N,l}(\xi)$ with
the inequalities having $l$ solutions for all large $X$. Accordingly, define
$w_{N,l}(\xi)$ and $\widehat{w}_{N,l}(\xi)$ for the linear form problem. Notice that
$\lambda_{N,1}(\xi)= \lambda_{N}(\xi)$ and  $\widehat{\lambda}_{N,1}(\xi)=
\widehat{\lambda}_{N}(\xi)$, as well as $w_{N,1}(\xi)= w_{N}(\xi)$ and
$\widehat{w}_{N,1}(\xi)= \widehat{w}_{N}(\xi)$ just recover the classical
exponents. As $\lambda$ can be chosen arbitrarily close to
$\widehat{\lambda}_{n}(\xi)$ in the first claim of Corollary~\ref{ccoo}, first asserts that, if~\eqref{eq:only} is satisfied, we have
\begin{equation} \label{eq:ceq}
\widehat{\lambda}_{n-m, m+1}(\xi) \geq \widehat{\lambda}_{n}(\xi)>\frac{1}{n-m+1}.
\end{equation}
Similarly the second claim upon ~\eqref{eq:only} reads
\begin{equation} \label{eq:ceq2}
\lambda_{n-m, m+1}(\xi) \geq \lambda_{n}(\xi)>\frac{1}{n-m+1}.
\end{equation}
These inequalities are important ingredients in the proof of
Theorem~\ref{openup2}.

\subsection{Parametric geometry of numbers} \label{pgn}
We give a very brief exposition of the concept of parametric geometry of
numbers due to Schmidt and Summerer~\cite{ss,ssmh}, where we only provide the
necessary results for this paper and refer to the quoted papers for more
details. Let $N\geq 1$ be an integer and $1\leq l\leq N+1$. Given $\xi$,
define $\psi_{N,l}(Q)$ as the supremum of exponents $\mu$ for which
\[
1\leq |x|\leq Q^{1+\mu}, \qquad \max_{1\leq i\leq N} |\xi^{i}x-y_{i}|\leq Q^{-1/N+\mu}
\]
has $l$ linearly independent integer vector solutions
$(x,y_{1},\ldots,y_{N})$. Let
\[
\underline{\psi}_{N,l}= \liminf_{Q\to\infty} \psi_{N,l}(Q), \qquad
\overline{\psi}_{N,l}= \limsup_{Q\to\infty} \psi_{N,l}(Q).
\]
Similarly, one can define the dual function $\psi_{N,l}^{\ast}(Q)$ as the
supremum of exponents $\mu$ such that the system of inequalities
$$
\max\{ |x_0|,|x_1|,\ldots, |x_N|\} \le Q^{1/N + \mu};\qquad |x_0+x_1\xi + \ldots + x_N\xi^N|<Q^{-1+\mu}
$$
has $l$ linearly independent integer vector solutions. The values
$\underline{\psi}_{N,l}^{\ast}, \overline{\psi}_{N,l}^{\ast}$ are then
defined analogously to $\underline{\psi}_{N,l}$ and $\overline{\psi}_{N,l}$.
As pointed out in~\cite[Equation~(4.11)]{ss}, Mahler's Duality Theorem on
Dual Convex bodies translates into
\begin{equation}  \label{eq:mahla}
\underline{\psi}_{N,l}^{\ast}= -\overline{\psi}_{N,N+2-l}, \qquad
\overline{\psi}_{N,l}^{\ast}= -\underline{\psi}_{N,N+2-l}.
\end{equation}
We further require the estimates from~\cite[(1.11)]{ssmh}:
\begin{equation} \label{eq:sses}
l\overline{\psi}_{N,l}+(N+1-l)\underline{\psi}_{N,N+1}\geq 0,\quad l\underline{\psi}_{N,l}+(N+1-l)\overline{\psi}_{N,N+1}\geq 0.
\end{equation}
and the relation from~~\cite[(1.15)]{ssmh}:
\begin{equation} \label{eq:untob}
\underline{\psi}_{N,l+1}\leq \overline{\psi}_{N,l}.
\end{equation}
To build a connection with Corollary~\ref{ccoo} we recall that these
quantities are related to the successive minima exponents $\lambda_{N,l},
\widehat{\lambda}_{N,l}$ from the previous section by identities. For
simplicity we drop the argument $\xi$ of the exponents in the following. All
claims below hold for any $1\leq l\leq N+1$. By~\cite[Theorem~1.4]{ss} (for
$l=1$, but the same argument also holds for larger $l$, see also~\cite{j2})
we have the identities
\[
(1+\lambda_{N,l})\cdot (1+\underline{\psi}_{N,l})= \frac{N+1}{N},
\]
and
\begin{equation}  \label{eq:fehlt2}
(1+\widehat{\lambda}_{N,l})\cdot (1+\overline{\psi}_{N,l})= \frac{N+1}{N}.
\end{equation}
Similarly for the dual linear form problem we have
\begin{equation}  \label{eq:fehlt3}
(1+w_{N,l})\cdot \left(\underline{\psi}_{N,l}^{\ast}+\frac{1}{N}\right)= \frac{N+1}{N},
\end{equation}
and
\begin{equation}  \label{eq:fehlt4}
(1+\widehat{w}_{N,l})\cdot \left(\overline{\psi}_{N,l}^{\ast}+\frac{1}{N}\right)= \frac{N+1}{N}.
\end{equation}

\section{Proof of Theorem~\ref{openup2}}
With aid of the results from Section~\ref{prepara}
we can prove Theorem~\ref{openup2}.

\begin{proof}[Proof of Theorem~\ref{openup2}]
    For all quotations of formulas in Section~\ref{pgn} below we let $N=n-m$.
    Note that the assumption \eqref{eq:only2} is stronger than \eqref{eq:only}.
    Therefore Corollary~\ref{ccoo} can be applied. Its claim
    \eqref{eq:ceq}, when combined with \eqref{eq:fehlt2} for $l=m+1$, implies
    \[
    \overline{\psi}_{n-m,m+1} \leq \frac{1-(n-m)\widehat{\lambda}_{n}(\xi)}{(n-m)(1+\widehat{\lambda}_{n}(\xi))}.
    \]
    On the other hand, since $2m<n$ by assumption, equations~\eqref{eq:mahla}
    and~\eqref{eq:sses} for $l=m+1$  give
    \[
    \overline{\psi}^{\ast}_{n-m,1}= - \underline{\psi}_{n-m,n-m+1} \leq \frac{m+1}{n-m+1-(m+1)}\cdot \overline{\psi}_{n-m,m+1}= \frac{m+1}{n-2m}\cdot \overline{\psi}_{n-m,m+1}.
    \]
    After inserting the above bound for $\overline{\psi}_{n-m,m+1}$ and
    simplifying the expression we get
    $$
    \overline{\psi}^{\ast}_{n-m,1} + \frac{1}{n-m} \le \frac{(n-m+1)(1-m\widehat{\lambda}_{n}(\xi))}{(n-m)(n-2m)(1+\widehat{\lambda}_{n}(\xi))}.
    $$
    Then \eqref{eq:h11} follows from~\eqref{eq:fehlt4} with $l=1$.

    Similarly, \eqref{eq:untob} yields
    \[
    \underline{\psi}_{n-m,m+2} \leq \overline{\psi}_{n-m,m+1} \leq \frac{1-(n-m)\widehat{\lambda}_{n}(\xi)}{(n-m)(1+\widehat{\lambda}_{n}(\xi))}
    \]
    and \eqref{eq:mahla}, \eqref{eq:sses} together with the assumption $2m+1<n$
    again implies
    \begin{align*}
    \underline{\psi}^{\ast}_{n-m,1}&= - \overline{\psi}_{n-m,n-m+1} \leq \frac{m+2}{n-m+1-(m+2)}\cdot \underline{\psi}_{n-m,m+2} \\
    &\leq \frac{m+2}{n-2m-1}\cdot\frac{1-(n-m)\widehat{\lambda}_{n}(\xi)}{(n-m)(1+\widehat{\lambda}_{n}(\xi))}.
    \end{align*}
    Then \eqref{eq:fehlt3} gives the stated left lower
    bound~\eqref{eq:h21} for $w_{n-m}(\xi)$. The right bound for $w_{n-m}(\xi)$
	follows similarly as \eqref{eq:h11} in view of \eqref{eq:ceq2} that is equivalent
	to the second claim of Corollary~\ref{ccoo}.

    Finally, \eqref{eq:windag} follows by combining \eqref{eq:h11}
    with the estimate
    \[
    \min\{ w_{m+1}(\xi), \widehat{w}_{n-m}(\xi)\} \leq \frac{1}{\widehat{\lambda}_{n}(\xi)}
    \]
    derived from~\cite[Theorem~2.1]{indag}. Indeed, we have $\widehat{\lambda}_{n}(\xi) \le 2/n<1/m$. Hence the assumptions~\eqref{eq:only2}
    and~\eqref{eq:h11} imply that $\widehat{w}_{n-m}(\xi)$ is larger than $1/\widehat{\lambda}_{n}(\xi)$, thus the left term in the minimum cannot exceed it.
\end{proof}

\section{Proof of Theorems~\ref{1B} and~\ref{1A}}

\subsection{Two relations between Diophantine exponents} \label{twor}

In this section we recall estimates linking $w_{n}^{\ast}$ with other
exponents of approximation. They will be required in the proofs of the main
results. Firstly, from~\cite[Theorem~2.7]{buschlei} any transcendental real
$\xi$ satisfies
\begin{equation} \label{eq:toll}
w_{n}^{\ast}(\xi) \geq \frac{3}{2}\widehat{w}_{n}(\xi)-n+\frac{1}{2}, \qquad n\geq 1.
\end{equation}
We will apply \eqref{eq:toll} for the index $n-m$ in context of Theorem~\ref{1A} for its proof. We lack analogues of \eqref{eq:toll}
for the modified versions of Wirsing's problem discussed at the
end of Section~\ref{int01}.
Therefore we cannot extend our results to these situations.
Secondly, a small variation of~\cite[Lemma~1]{davsh} implies the relation
\begin{equation} \label{eq:davschm}
w_{n}^{\ast}(\xi) \geq \frac{1}{\widehat{\lambda}_{n}(\xi)}, \qquad n\geq 1.
\end{equation}
In fact, Lemma~1
from~\cite{davsh} provides a lower estimate for $w_{n+1}^{\ast}(\xi)$ instead
of $w_{n}^{\ast}(\xi)$, however it is well-known to hold for the latter as
well, see for example~\cite{j2} or~\cite{buteu}.

For the sake of completeness, we state the related inequalities
\begin{equation}  \label{eq:toller}
w_{n}^{\ast}(\xi) \geq \frac{w_{n}(\xi)+1}{2}, \qquad
w_{n}^{\ast}(\xi) \geq w_{n}(\xi)-n+1, \qquad w_{n}^{\ast}(\xi) \geq \frac{\widehat{w}_{n}(\xi)}{\widehat{w}_{n}(\xi)-n+1}
\end{equation}
by Wirsing~\cite{wirsing}, Bugeaud~\cite[Lemma~1A]{bugbuch} and Bugeaud and
Laurent~\cite{bula} respectively.
Many of the above inequalities directly show the lower bounds
\[
\overline{w}^{\ast}(\xi) \geq \underline{w}^{\ast}(\xi) \geq \frac{1}{2}.
\]
However, as indicated in the introduction, no improvement of the
constant $1/2$ even for the larger quantity
$\overline{w}^{\ast}(\xi)$ seems obvious from previous results.
For our method, to improve $\underline{w}^{\ast}(\xi) \geq \frac{1}{2}$
it is essential to use
\eqref{eq:toll}, the bounds in \eqref{eq:toller} are insufficient.
On the other hand, the left estimate in \eqref{eq:toller}
would imply $\overline{w}^{\ast}(\xi)>1/2$ when utilized in the framework below.

\subsection{Deduction of the main results}

\begin{proof}[Proof of Theorem~\ref{1A}]
    Let $m,n$ and $\xi$ be as in the theorem.
    First assume inequality \eqref{eq:only} holds. Then we apply Theorem~\ref{openup2}
    which together with \eqref{eq:toll} for index $n-m$ yields
    \begin{equation} \label{eq:kuerzer}
    w_{n}^{\ast}(\xi)\geq w_{n-m}^{\ast}(\xi)\geq
    \frac{3}{2}\cdot\frac{(n-m)\widehat{\lambda}_{n}(\xi)+n-2m-1}{1-m\widehat{\lambda}_{n}(\xi)}-(n-m)+\frac{1}{2}.
    \end{equation}

    Denote the right hand side by $\tau=\tau_{m,n}(\widehat{\lambda}_{n}(\xi))$.
    Regardless if \eqref{eq:only} holds or not,
    the estimates \eqref{eq:davschm} and~\eqref{eq:kuerzer}
    together imply
    \begin{equation}\label{eq:eq23}
    w_{n}^{\ast}(\xi)\geq \max\left\{ \tau\cdot \underline{1}_{\left(\frac{1}{n-m+1},1\right)}(\widehat{\lambda}_{n}(\xi)),\; \frac{1}{\widehat{\lambda}_{n}(\xi)}\right\},
    \end{equation}
    where $\underline{1}_{I}(t)$ denotes the indicator function of an interval $I$. The first term in the maximum is rising as a function of $\widehat{\lambda}_{n}(\xi)$ on $[1/n,1/m)$, while the second term is obviously decreasing.
    It is easy to check that for $\widehat{\lambda}_{n}(\xi)=1/(n-m+1)$
    and slightly larger values
    the right term prevails (since then $\tau>1/\widehat{\lambda}(\xi)$),
    while for
    $\widehat{\lambda}_{n}(\xi)=1/m$ the left term becomes bigger (it
    actually tends monotonically to infinity). Therefore the minimum of the right hand side
    of~\eqref{eq:eq23} is attained when the expressions are equal. This happens when $\widehat{\lambda}_{n}(\xi)$ solves the quadratic equation in $\lambda$:
    \begin{equation}\label{eq24}
    (2mn-2m^2 +3n-4m)\lambda^2 + (n-2m-2)\lambda -2=0.
    \end{equation}
    The reciprocal of this equilibrium
    value, according to \eqref{eq:eq23}, can readily be calculated as the lower bound in
    Theorem~~\ref{1A}.
    \end{proof}

    %
    %
    %

The lower bound in Theorem~\ref{1A} can be slightly improved if instead
of~\eqref{eq:toll} one uses the stronger estimate
\[
w_{n}^{\ast}(\xi) \geq \frac{w_{n}(\xi)}{2}+\widehat{w}_{n}(\xi)-n+\frac{1}{2}, \qquad n\geq 1,
\]
which holds as soon as $w_{n}^{\ast}(\xi)\leq n$. The last inequality can be
derived by applying the proof in~\cite{buschlei}. Using the left bound in
\eqref{eq:h21} we indeed obtain an improvement.

Now, Theorem~\ref{1B} follows from Theorem~\ref{1A} with a proper choice of
the parameter $m$.

\begin{proof}[Proof of Theorem~\ref{1B}]
    Write
    \[
    \Phi(u,v)= \frac{4uv+6v-4u^{2}-8u}{2u+2-v+\sqrt{v^{2}+12uv+20v-12u^{2}-24u+4}}
    \]
    so that $\Phi(m,n)$ is the bound in Theorem~\ref{1A}.
If we fix $v$ and write $u=\alpha v$ for $\alpha\in(0,1/2)$, then we obtain a
bound of order $\Phi(u,v)\geq F(\alpha)v+o(v)$ as $v\to\infty$, for the
function
\[
F(t)= \frac{4t-4t^{2}}{2t-1+\sqrt{1+12t-12t^{2}}}.
\]
By differentiation one can check that $F(t)$ is maximized for
$\alpha=\alpha_{0}:=(3-\sqrt{3})/6=0.2113\ldots$ with maximum
$F(\alpha_{0})=\beta:=1/\sqrt{3}$. For given $v=n$, if we take $m=\lfloor
n\alpha_0\rfloor$ then the quotient $\Phi(m,n)/n$ will be arbitrarily close
to $\Phi(\alpha_{0} n,n)/n$ for large enough $n$. By Theorem~\ref{1A} and
continuity of $F$ we infer $w_{n}^{\ast}(\xi)/n\geq
F(\alpha_{0})-o(1)=\beta-o(1)$ as $n\to\infty$.

Next, we need to show that for all $n\ge 4$ there exists $m<(n-1)/2$ such
that $\Phi(m,n)/n$ exceeds $\beta$. This is equivalent to saying that for the
same values $n$ and $m$ the solution $\lambda$ of~\eqref{eq24} is less than
$\sqrt{3}/n$. By substituting $\lambda = \sqrt{3}/n$ into the left hand side
of the equation we get
$$
-\left(\frac{3-\sqrt{3}}{\sqrt{6}}n -
\sqrt{6} m\right)^2 + (9-2\sqrt{3})n- 12m.
$$
For $m=\lfloor n\alpha_0\rfloor$, the square part of this expression is at
least $-6$ while the remaining part is at least
$$
-2(3-\sqrt{3})n + (9-2\sqrt{3})n = 3n \ge 12.
$$
Therefore the expression is positive and therefore for $m=\lfloor
n\alpha_0\rfloor$, the values $\Phi(m,n)/n$ are larger than $\beta$.

We finally settle \eqref{eq:obenba}. We will show that for any $\epsilon>0$
and large $n\geq n_{0}(\epsilon)$, with $m=\lfloor\gamma n\rfloor$ for a
certain $\gamma$  and $s=n-m$ we have
\begin{equation} \label{eq:reicht}
\max\left\{ \frac{w_{n}^{\ast}(\xi)}{n}, \frac{w_{s}^{\ast}(\xi)}{s}\right\} > \delta-\epsilon,
\end{equation}
with $\delta$ defined in the theorem.
This clearly implies the claim.

Let $\gamma\in(0,1/2)$ be a parameter. Choose
 $m=\lfloor\gamma n\rfloor$ and denote
 $c=n\widehat{\lambda}_{n}(\xi)$. Clearly $c\in[1,2]$ by Dirichlet's Theorem and~\cite{laurent}.
On the one hand, \eqref{eq:davschm} implies $w_{n}^{\ast}(\xi)/n\geq c^{-1}$
, on the other hand by our choice of $\gamma$ we may apply
Theorem~\ref{openup2} and again derive a similar estimate to
\eqref{eq:kuerzer}. Putting negligible terms in a remainder term yields
\[
\frac{w_{s}^{\ast}(\xi)}{s}=\frac{w_{n-m}^{\ast}(\xi)}{n-m}\geq \frac{3}{2}\cdot \left(\frac{1-2\gamma}{(1-\gamma)(1-c\gamma)}\right)-1-o(1),
\qquad n\to\infty.
\]
Thus for every parameter $\gamma\in(0,1/2)$ we have
\[
\max\left\{ \frac{w_{n}^{\ast}(\xi)}{n}, \frac{w_{s}^{\ast}(\xi)}{s}\right\} \geq \min_{c\in[1,2]} \max\left\{\frac{1}{c},\; \frac{3}{2}\cdot \left(\frac{1-2\gamma}{(1-\gamma)(1-c\gamma)}\right)-1 \right\} -o(1),
\]
as $n\to\infty$. For given $\gamma$ the minimum of
the inner
maximum is obtained
when the expressions are equal, that is for
\[
c=c(\gamma)=\frac{2\gamma^{2}+2\gamma-1+\sqrt{4\gamma^{4}+24\gamma^{3}-32\gamma^{2}+12\gamma+1}}{4(\gamma-\gamma^{2})}
\]
obtained as a solution of a quadratic equation. Observe that the right hand
side is $1/G(\gamma)$ with $G$ defined in \eqref{eq:G}. Matlab calculations
show that the reciprocal $1/c(\gamma)$ is maximized over $\gamma\in(0,1/2)$
for a numerical value $\gamma_{0}=0.2345\ldots$ which by differentiation can
be checked to be a root of the irreducible quartic $Q(t)=4t^4 - 12t^3 + 10t^2
- 6t + 1$, yielding a bound $\delta=G(\gamma_{0})>0.6408$ thereby
verifying~\eqref{eq:reicht}.
\end{proof}

$\boldsymbol{Acknowledgements}$: Foundations of the results in this paper
were established at the collaborative
workshop ''Ergodic theory, Diophantine approximation and related topics'' at MATRIX institute in Creswick, Australia in June 2019.

\bigskip
\noindent Dzmitry Badziahin\\ \noindent The University of Sydney\\
\noindent Camperdown 2006, NSW (Australia)\\
\noindent {\tt dzmitry.badziahin@sydney.edu.au}

\bigskip
\noindent Johannes Schleischitz\\ \noindent Middle East Technical
University\\
\noindent Northern Cyprus Campus\\
\noindent Kalkanlı, Güzelyurt, KKTC\\
\noindent via Mersin 10, Turkey\\
\noindent{\tt johannes.schleischitz@univie.ac.at}

\end{document}